\documentclass[oneside]{article}
\usepackage{setspace}
\usepackage{titlesec}
\usepackage[utf8]{inputenc}
\usepackage{tabularx}	
\usepackage{etoolbox}
\usepackage{mathtools}
\usepackage{amsmath}
\usepackage{amssymb} 
\usepackage{amsthm,pifont}
\usepackage{graphicx}
\usepackage{bbm}
\usepackage{hyperref}
\usepackage[nottoc,numbib]{tocbibind}
\usepackage{xcolor}
\usepackage{tikz}

\usepackage{cite}
\usepackage{braket}
\usepackage{leftidx,tensor}
\usepackage{mathtools}
\usepackage{dsfont}
\usepackage{calc}
\usepackage{caption}
\usepackage{subcaption}
\usepackage{setspace}
\usepackage[printonlyused,nohyperlinks]{acronym}

\usetikzlibrary{decorations.markings}
\usetikzlibrary{shapes.geometric}
\tikzstyle{vertex}=[circle, draw, inner sep=0pt, minimum size=6pt]
\newcommand{\vertex}{\node[vertex]}

\usepackage[utf8]{inputenc}
\usepackage[english]{babel}

\newtheorem{theorem}{Theorem}[section]

\makeatletter
\patchcmd{\ttlh@hang}{\parindent\z@}{\parindent\z@\leavevmode}{}{}
\patchcmd{\ttlh@hang}{\noindent}{}{}{}
\makeatother

\titleformat*{\section}{\large\bfseries}
\titleformat*{\subsection}{\small\bfseries}
\titleformat*{\subsubsection}{\large\bfseries}
\titleformat*{\paragraph}{\large\bfseries}
\titleformat*{\subparagraph}{\large\bfseries}

\newcommand{\N}{\mathbb{N}}
\newcommand{\R}{\mathbb{R}}

\newtheorem{lemma}[theorem]{Lemma}

\newtheorem{corollary}[theorem]{Corollary}

\usepackage{geometry}
\begin{document}
	
	\begin{center}
		{\huge Uniqueness and Non-Uniqueness for Spin-Glass Ground States on Trees}\\\vspace{0.4cm}
		Johannes Bäumler\footnote{\textsc{Department of Mathematics, TU Munich}. E-Mail: \href{mailto:johannes.baeumler@tum.de}{johannes.baeumler@tum.de}}
	\end{center}
	
		\noindent
		\textbf{Abstract.} We consider a spin glass at temperature $T = 0$ where the underlying graph is a locally finite tree. We prove for a wide range of coupling distributions that uniqueness of ground states is equivalent to the maximal flow from any vertex to $\infty$ (where each edge $e$ has capacity $|J_{e}|$) being equal to zero which is equivalent to recurrence of the simple random walk on the tree.
	
	\tableofcontents
	
	\vspace{1cm}

\section{Introduction and Definitions}
Let $G = (V,E)$ be a locally finite graph, for a given finite set $B \subset V$ define $E(B)$ as the set of edges with at least one end in $B$. For any finite set $B \subset V$, $\sigma \in \left\{-1,+1\right\}^V$ and $J_{B} \coloneqq (J_{xy}, (x,y)\in E(B))$ define
\begin{equation}\label{ham}
H_{B,J}= \space -\sum_{(x,y) \in E(B)}J_{xy}\sigma_{x}\sigma_{y} \ .
\end{equation}
In the Edwards-Anderson spin glass model \cite{edwards1975theory} one considers nearest neighbor interactions and the case where the $J_{xy}$$'s$, also called the {\sl couplings}, are i.i.d. random variables. The distribution of a single coupling will be denoted by $\nu$, throughout we assume that $\nu(\{0\})=0$. With $\nu^{E}$ we denote the distribution of $(J_{xy}, (x,y)\in E)$.\newline
We call an edge $e=(x,y)\in E$ {\sl satisfied} for a configuration $\sigma$ if $J_{xy}\sigma_{x}\sigma_{y}>0$, if $J_{xy}\sigma_{x}\sigma_{y}<0$ we call it {\sl unsatisfied}. Note that $\nu^{E}( \exists  e \in E $ such that $ J_{xy}=0)=0$. \newline
Ground states are local minima of the Hamiltonian defined in (\ref{ham}), i.e. the Hamiltonian (\ref{ham}) can not be lowered by flipping the spins for some finite $B \subset V$. This means that $\sigma$ is a ground state if and only if for any finite set $B \subset V$
\begin{equation}\label{gs}
\sum_{(x,y) \in \partial B}J_{xy}\sigma_{x}\sigma_{y} \geq 0
\end{equation}
where $\partial B$ denotes the set of edges with exactly one end in $B$, see also \cite{bolthausen2007spin} or \cite{newman2012topics} for a more general introduction to spin glasses.
We denote the set of ground states of $G$ with couplings $J$ by $\mathcal{G}(J)$. 
The main goal of this paper is to determine the cardinality of $\mathcal{G}(J)$, when the underlying graph $G = (V,E)$ is a tree and $J$ is distributed according to some distribution satisfying $\nu((-\epsilon,\epsilon))=\Theta(\epsilon)$ for $\epsilon \rightarrow 0$; that means there exist $0<c<C<\infty$ such that $c\cdot\epsilon\leq\nu((-\epsilon,\epsilon))\leq C\cdot\epsilon$ for all $\epsilon$ small enough. A distribution $\nu$ satisfying this will also be called a {\sl distribution of linear growth}.
The existence of ground states is a consequence of compactness of the space $\left\{-1,+1\right\}^{V}$. Clearly $\sigma$ is a ground state if and only if $-\sigma$ is a ground state, hence $|\mathcal{G}(J)| $ is even (or infinity) and greater or equal than 2. For trees there are two {\sl natural  ground  states}, namely the ones satisfying $ J_{xy}\sigma_{x}\sigma_{y} \geq 0$ for every $(x,y)\in E$. In Theorem \ref{T1} in Section 2 we will see a necessary and sufficient condition which ensures that the two natural ground states are the only ground states. There are already many results about the behavior of spin glasses on trees \cite{carlson1988critical,carlson1990bethe,gandolfo2017glassy} and its ground states \cite{tessler2010geometry}, but most of them are limited to trees with a high regularity, for example the Bethe Lattice.\\

For a tree $T = (V,E)$ we choose one vertex and call it the root of the tree or simply 0. All concepts presented in the following are independent of the specific choice of the root. For $x \in V$ let $|x|$ be the length of the shortest (and hence only non intersecting) path starting from zero and ending in $x$. For $e = (x,y) \in E $ $|e| \coloneqq \min \ \left\{ |x|,|y| \right\}$.\newline
By $x \preceq y$ we mean that $x$ is part of every path connecting 0 and $y$. This directly implies that $|x|\leq|y|$. Note that $0 \preceq y$ $\forall y \in V$. We say that $x\rightarrow y$ if $x \preceq y$ and $|x| = |y| - 1$, i.e. $(x,y)\in E$ and $x$ is the vertex located closer to the root.\newline
For $x \in V$ we define the subgraph $T_{\succeq x}$ by the tree containing all vertices $y$, s.t. $x \preceq y$ and all edges of the form $(u,v)$, s.t. $(u,v) \in E, x \preceq u$ and $x\preceq v$.\newline
For $n \in \N$ we define $T_{\leq n} = (V_{\leq n},E_{\leq n}) $ by the tree containing all vertices $x$ s.t. $|x|\leq n$ and and edges $e$ s.t. $|e| \leq n - 1$. By $T_{\geq n} = (V_{\geq n},E_{\geq n})$ we mean the forest containing all vertices $x$ s.t. $|x|\geq n$ and and edges $e$ s.t. $|e| \ge n$.\newline
Throughout we will assume that all edges are oriented towards infinity, i.e. that $(x,y) \in E$ implies $x \preceq y$.\newline
For some edge $e=(x,y)$, we denote the shortest path connecting the root 0 to $y$ by $\mathcal{P}_{e}$.\\

A subset $\Pi \subset E$ is called a {\sl cutset} separating $x$ and infinity if every infinite non self-intersecting path starting at $x$ contains at least one edge in $\Pi$. If we do not mention a specific vertex $x$, we always mean that the cutset separates 0 from $\infty$. For a function $g : \left\{F \subset E : |F|<\infty \right\} \rightarrow \R$ we set
\begin{equation*}
\liminf_{\Pi \rightarrow \infty} g(\Pi) \coloneqq \lim_{n \rightarrow \infty}
\inf \left\{g(\Pi) : \Pi \ \mathrm{ cutset} \wedge \Pi \subset E_{\geq n}\right\} \ .
\end{equation*}

\section{The main theorem}    

The main goal of this section is to prove the following theorem:

\begin{theorem}\label{T1}
	Let $\nu$ be a distribution of linear growth. Then the following are equivalent: \newline
	i) The natural ground states are the only ground states $\nu^{E}$-a.s. \newline
	ii) inf $ \{ \  \sum_{e\in \Pi} |J_{e}|  \ : \ \Pi $ cutset separating 0 and $ \infty \} = 0$ $\nu^{E}$-a.s.\newline
	iii) MaxFlow$(0 \rightarrow \infty, \langle |J_{e}| \rangle) = 0 \ \nu^{E}$-a.s.\newline
	iv) The simple random walk on T = (V,E) is recurrent
\end{theorem}

The equivalence of $i)$ and $ii)$ will be proven in section \ref{2.1}, the equivalence of $ii)$ and $iii)$ is just a well known extension of the Max-Flow Min-Cut - Theorem of L.R. Ford and D.R. Fulkerson \cite{ford2015flows}. In the sections \ref{2.2} and \ref{2.3} we will show the equivalence of $iii)$ and $iv)$. In sections \ref{2.1}-\ref{2.3} we will for technical reasons assume that $T$ is a tree without finite branches, that means that $T_{\succeq x}$ is infinite for every $x \in V$. In section \ref{2.4} we will prove that this assumption is in fact not necessary. 

\subsection{Spin Glasses, Cutsets and Flows}\label{2.1}

\begin{figure}
	\[\begin{tikzpicture}[scale = 1.2,- ,>= latex ]
	\vertex[fill] (0) at (0,0) [label=below:$0$] {};
	
	\vertex (01) at (-0.6,1) [line width=0.75pt] {};
	\vertex[fill] (02) at (1.2,1) {};
	
	\vertex (011) at (-1.6,2) [line width=0.75pt] {};
	\vertex[fill] (012) at (-0.6,2) {};
	\vertex (013) at (0.1,2) [line width=0.75pt] {};
	\vertex[fill] (021) at (0.9,2) {};
	\vertex[fill] (022) at (1.5,2) {};
	
	\vertex[fill] (0111) at (-1.9,3) {};
	\vertex[fill] (0112) at (-1.3,3) {};
	\vertex[fill] (0121) at (-0.6,3) {};
	\vertex[fill] (0131) at (0.1,3) {};
	\vertex[fill] (0211) at (0.9,3) {};
	\vertex[fill] (0221) at (1.4,3) {};
	\vertex[fill] (0222) at (1.9,3) {};

	\path
	
	(0) edge [line width=1.5pt] node[left=0] {$e$} (01)
	(0) edge (02)
	
	(01) edge (011)
	(01) edge [dashed, line width=1.3pt] (012)
	(01) edge (013)
	(02) edge (021)
	(02) edge [dashed] (022)
	
	(011) edge [dashed, line width=1.3pt] (0111)
	(011) edge [dashed, line width=1.3pt] (0112)
	(012) edge (0121)
	(013) edge [dashed, line width=1.3pt] (0131)
	(021) edge [dashed] (0211)
	(022) edge (0221)
	(022) edge (0222)

	;
	\end{tikzpicture}\]
	\centering
	\caption{$e$ is bold, $\Pi$ dashed, $\Pi_{\succeq e}$ bold and dashed, the set $B \subset V$ are the blank vertices.}
	\label{Fig1}
	
\end{figure}

A function $\theta : E \rightarrow \R_{\ge 0}$ is called a {\sl flow from $0$ to $\infty$}, or just {\sl flow}, if for all $y \in V\setminus \{$0$\}$
\begin{equation*}
\sum_{\{x \in V : x \rightarrow y\}}\theta ((x,y)) = \sum_{\{z \in V : y \rightarrow z\}}\theta ((y,z)) \ .
\end{equation*}
$\theta$ is a flow with respect to the capacities $\kappa_{e}$ if additionally
\begin{equation*}
\theta(e) \leq \kappa_{e}
\end{equation*}
holds for all $e \in E$. Most of the time we will deal with the case where the capacities are $|J_{e}|$. For a flow $\theta$ we define the {\sl strength} of the flow as
\begin{equation*}
\mbox{strength}(\theta) \coloneqq \sum_{\{x \in V : 0 \rightarrow x\}}\theta ((0,x)) \ .
\end{equation*}
For a given set of nonnegative values $(\kappa_{e})_{e \in E}$ we define the {\sl maximal flow} from 0 to $\infty$ as
\begin{equation*}
\mbox{MaxFlow} (0 \rightarrow \infty, \langle \kappa_{e} \rangle) \coloneqq
\max\{\mbox{strength}(\theta) : \theta \ \mathrm{ is \ a \ flow \ w.r.t. \ } \kappa_{e} \} \ .
\end{equation*}
Remember that due to the enhanced version of the Max-Flow Min-Cut Theorem, see for example (\cite{lyons2017probability}, Chapter 3)
\begin{equation}\label{MFMC}
\mbox{MaxFlow} (0 \rightarrow \infty, \langle |J_{e}| \rangle) = \inf\left\{ \sum_{e\in \Pi} |J_{e}| : \Pi  \mathrm{ \ cutset \ separating \ 0 \ and \ }  \infty \right\} \ .
\end{equation}
\noindent
With this, we are ready to prove the equivalence of $i)$ and $ii)$ of Theorem \ref{T1}.
\begin{proof}
	$ii) \Rightarrow i)$: Let $\sigma \in \left\{-1,+1\right\}^{V}$ such that there exists some edge $e = (x,y) \in E$ s.t. $J_{e}\sigma_{x}\sigma_{y} < 0$. Take a cutset $\Pi$ such that $\sum_{f \in \Pi}|J_{f}| < |J_{e}|$ and $e$ lies on the same side as the root, a picture of such a situation is given in Figure \ref{Fig1}. Such a cutset $\Pi$ exists almost surely by the assumption $ii)$. Define $\Pi_{\succeq e}$ as all elements $f\in \Pi$ satisfying $f \succeq e$. Now let $B \subset V$ be the set of all vertices which are separated from infinity by $\left\{e\right\}\cup \Pi_{\succeq e}$. Then
	\begin{equation*}
	\sum_{(x,y) \in \partial B}J_{xy}\sigma_{x}\sigma_{y} = \sum_{(x,y) \in \left\{e\right\}\cup \Pi_{\succeq e}}J_{xy}\sigma_{x}\sigma_{y} < 0 \ .
	\end{equation*}
	
	\begin{figure}
		\begin{center}

			\begin{subfigure}{.3\linewidth}
				\[\begin{tikzpicture}[scale = 1., y={1cm/1.2}]
				\vertex (0) at (0,0) [line width=0.75pt, label=below:$0$] {};
				
				\vertex (11) at (-0.75,1) [line width=0.75pt, label=right:$u$] {};
				\vertex (12) at (0.6,1) [line width=0.75pt] {};
				
				\vertex (21) at (-1.2,2) [line width=0.75pt, label=left:$v$] {};
				\vertex[fill] (22) at (1.2,2) {};
				
				\vertex[fill] (31) at (-1.6,3) {};
				\vertex (32) at (-0.8,3) [line width=0.75pt] {};
				\vertex[fill] (33) at (0.6,3) {};
				\vertex[fill] (34) at (1.6,3) {};
				
				\vertex[fill] (41) at (-1.6,4) {};
				\vertex[fill] (42) at (-0.8,4) {};
				\vertex[fill] (43) at (0.5,4) {};
				\vertex[fill] (44) at (1.2,4) {};
				\vertex[fill] (45) at (1.8,4) {};
				
				\vertex[draw=none] (i1) at (0.1,0.4) {};
				\vertex[draw=none] (i2) at (-0.5,4.4) {};
				
				\path
				(0) edge (11)
				(0) edge (12)
				
				(11) edge node[left=0] {$h$} (21)
				(12) edge [dashed] (22)
				
				(21) edge [dashed] (31)
				(21) edge (32)
				(22) edge (33)
				(22) edge (34)
				
				(31) edge (41)
				(32) edge [dashed] (42)
				(33) edge (43)
				(33) edge (44)
				(34) edge (45)
				
				(i1) edge[draw=lightgray, bend right=5] node[left] {\textcolor{gray}{$T_2$}} node[right] {\textcolor{gray}{$T_1$}} (i2)
				;
				\end{tikzpicture}\]
			\end{subfigure}
			\begin{subfigure}{.3\linewidth}
				\[\begin{tikzpicture}[scale = 1., y={1cm/1.2}]
				\vertex (0) at (0,0) [line width=0.75pt, label=below:$0$] {};
				
				\vertex (11) at (-0.75,1) [line width=0.75pt, label=right:$u$] {};
				\vertex (12) at (0.6,1) [line width=0.75pt] {};
				
				\vertex[fill] (21) at (-1.2,2) [label=left:$v$] {};
				\vertex (22) at (1.2,2) [line width=0.75pt] {};
				
				\vertex[fill] (31) at (-1.6,3) {};
				\vertex[fill] (32) at (-0.8,3) {};
				\vertex[fill] (33) at (0.6,3) {};
				\vertex (34) at (1.6,3) [line width=0.75pt] {};
				
				\vertex[fill] (41) at (-1.6,4) {};
				\vertex[fill] (42) at (-0.8,4) {};
				\vertex[fill] (43) at (0.5,4) {};
				\vertex[fill] (44) at (1.2,4) {};
				\vertex[fill] (45) at (1.8,4) {};
				
				\vertex[draw=none] (i1) at (0.1,0.4) {};
				\vertex[draw=none] (i2) at (-0.5,4.4) {};
				
				\path
				(0) edge (11)
				(0) edge (12)
				
				(11) edge [dashed] node[left=0] {$h$} (21)
				(12) edge (22)
				
				(21) edge (31)
				(21) edge (32)
				(22) edge [dashed] (33)
				(22) edge (34)
				
				(31) edge (41)
				(32) edge (42)
				(33) edge (43)
				(33) edge (44)
				(34) edge [dashed] (45)
				
				(i1) edge[draw=lightgray, bend right=5] node[left] {\textcolor{gray}{$T_2$}} node[right] {\textcolor{gray}{$T_1$}} (i2)
				;
				\end{tikzpicture}\]
			\end{subfigure}
			\begin{subfigure}{.3\linewidth}
				\[\begin{tikzpicture}[scale = 1., y={1cm/1.2}]
				\vertex[fill] (0) at (0,0) [label=below:$0$] {};
				
				\vertex[fill] (11) at (-0.75,1) [label=right:$u$] {};
				\vertex[fill] (12) at (0.6,1) {};
				
				\vertex (21) at (-1.2,2) [line width=0.75pt, label=left:$v$] {};
				\vertex[fill] (22) at (1.2,2) {};
				
				\vertex[fill] (31) at (-1.6,3) {};
				\vertex (32) at (-0.8,3) [line width=0.75pt] {};
				\vertex[fill] (33) at (0.6,3) {};
				\vertex[fill] (34) at (1.6,3) {};
				
				\vertex[fill] (41) at (-1.6,4) {};
				\vertex[fill] (42) at (-0.8,4) {};
				\vertex[fill] (43) at (0.5,4) {};
				\vertex[fill] (44) at (1.2,4) {};
				\vertex[fill] (45) at (1.8,4) {};
				
				\vertex[draw=none] (i1) at (0.1,0.4) {};
				\vertex[draw=none] (i2) at (-0.5,4.4) {};
				
				\path
				(0) edge (11)
				(0) edge (12)
				
				(11) edge [dashed] node[left=0] {$h$} (21)
				(12) edge (22)
				
				(21) edge [dashed] (31)
				(21) edge (32)
				(22) edge (33)
				(22) edge (34)
				
				(31) edge (41)
				(32) edge [dashed] (42)
				(33) edge (43)
				(33) edge (44)
				(34) edge (45)
				
				(i1) edge[draw=lightgray, bend right=5] node[left] {\textcolor{gray}{$T_2$}} node[right] {\textcolor{gray}{$T_1$}} (i2)
				;
				\end{tikzpicture}\]
			\end{subfigure}
			
		\end{center}
		
		\caption{$\partial B$ are the dashed edges, the blank vertices are the finite sets $B \subset V$. The grey line shows the separation between $T_1$ and $T_2$}
		\label{Fig2}
	\end{figure}
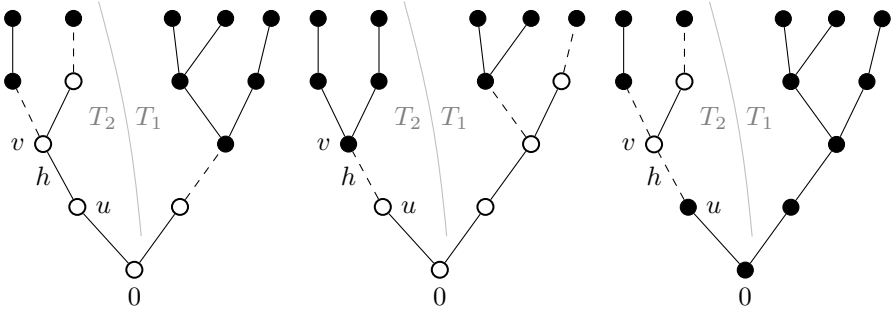
	
	So $\sigma$ is not a ground state.\\
	
	\noindent
	$i) \Rightarrow ii)$: We assume that $ \inf\{\sum_{e\in \Pi} |J_{e}| : \Pi \ \mbox{cutset}\} > 0$ and construct a non-trivial ground state from this. For a tree $T$ let $T_{1} = (V_{1},E_{1}), \ T_{2} = (V_{2},E_{2})$ be two subtrees such that $V_{1} \cap V_{2} = \left\{0\right\}$ and $\mbox{MaxFlow}_{T_{1}}(0 \rightarrow \infty, \langle |J_{e}| \rangle) \geq \mbox{MaxFlow}_{T_{2}}(0 \rightarrow \infty, \langle |J_{e}| \rangle) > 0$, where $\mbox{MaxFlow}_{T_{1}}$ is the maximal flow in $T_{1}$, respectively $T_{2}$. We choose the subtrees $T_{1}$ and $T_{2}$ before we choose the root here. Then we can orient all edges pointing away from 0. Let $h = (u,v) \in E_{2}$, such that for every $\epsilon > 0$
	\begin{equation*}
	\mbox{MaxFlow}(0 \rightarrow \infty, \langle |J_{e}| \rangle) <
	\mbox{MaxFlow}(0 \rightarrow \infty, \langle |J_{e}| + \epsilon \mathbbm{1}_{\left\{e=h\right\}} \rangle) \ .
	\end{equation*}
	That means that increasing the coupling value at $h$ increases the Maximal Flow from the root to $\infty$ in $T$ and therefore also in $T_{2}$. In particular $|J_{h}| \leq  \mbox{MaxFlow}_{T_{2}}(0 \rightarrow \infty, \langle |J_{e}| \rangle)$; the almost sure existence of such an edge and such subtrees will be discussed in Lemma \ref{L2}.\newline
	\noindent
	Now define $\sigma \in \left\{-1,+1\right\}^{V}$ by
	\begin{equation}\label{sigma}
	J_{xy}\sigma_{x}\sigma_{y} =
	\begin{cases}
	< 0 & (x,y)=h \\
	> 0 & \mbox{else}
	\end{cases}
	\end{equation}
	for $(x,y) \in E$. From (\ref{sigma}) one can extract a $\sigma \in \{-1,+1\}^V$, which is unique up to a global spin flip and further $\sigma$ is a ground state. To see this we will show (\ref{gs}) for three different cases of finite sets $B \subset V$. The three cases correspond to the three trees (from left to right) in Figure \ref{Fig2}.\\
	\noindent
	{\sl Case 1.} $(u,v) \notin \partial B$ : Here (\ref{gs}) is clearly true, as $(u,v) \notin \partial B$ and $J_{xy}\sigma_{x}\sigma_{y} \geq 0$ for every $(x,y) \in E, (x,y) \neq (u,v)$.\\
	\noindent
	{\sl Case 2.} $u \in B, v \notin B$ : First note that $|J_{h}| = \min\left\{|J_{f}| : f \in \mathcal{P}_{h} \right\}$ due to the construction. If $\partial B$ contains an edge $f \in \mathcal{P}_{h}$ (\ref{gs}) is true, as $|J_{h}| \leq |J_{f}|$. Otherwise $\partial B$ contains a cutset $\Pi$ separating 0 and $\infty$ in $T_{1}$. Hence 
	\begin{equation*}|J_{h}| \leq \mbox{MaxFlow}_{T_{1}}(0 \rightarrow\ \infty, \langle |J_{e}| \rangle) \leq \sum_{f \in \Pi} |J_{f}| \end{equation*} 
	and (\ref{gs}) is true.  \\
	\noindent
	{\sl Case 3.} $u \notin B, v \in B$ : Here $\partial B$ contains a cutset $\Pi_{\succeq h}$ separating $h$ and $\infty$ in $T_{\succeq h}$. For this cutset
	\begin{equation*}
	\sum_{e \in \Pi_{\succeq h}} |J_{e}| \geq |J_{h}|
	\end{equation*}
	due to (\ref{MFMC}). So 
	\begin{equation*}
	\sum_{(x,y) \in \partial B} J_{xy}\sigma_{x}\sigma_{y} \geq \sum_{e \in \Pi_{\succeq h}} |J_{e}| - |J_{h}| \geq 0
	\end{equation*}
	and $\sigma$ is a ground state.
\end{proof}

As vanishing of the maximal flow from the root to $\infty$ does not depend on the values of finitely many couplings (we assumed $\nu(\left\{0\right\})=0$) we get that $\nu^{E}(MaxFlow(0 \rightarrow \infty, \langle |J_{e}| \rangle) = 0) \in \left\{0,1\right\}$ by Kolmogorov's 0-1-law. Hence uniqueness of ground states is a deterministic property for trees with coupling distributions of linear growth; in Corollary \ref{infgs} we will see that even $|\mathcal{G}(J)|=\infty$ in the case of non-uniqueness holds. The items $ii)$ and $iii)$ of Theorem $\ref{T1}$ are equivalent for every graph by the MaxFlow-MinCut-Theorem \cite{ford2015flows}. For the proof of the implication $ii) \Rightarrow i)$ in Theorem \ref{T1} we did not use the linear growth assumption, so this implication holds, whenever $\nu(\left\{0\right\})=0.$

\begin{lemma}\label{L2}
	Let $T=(V,E)$ be a tree and $J_{e}$ be distributed according to some distribution of linear growth $\nu$ such that MaxFlow$(0\rightarrow \infty, \langle |J_{e}| \rangle) > 0$ a.s., then there exists some vertex $0$ and subtrees $T_{1}=(V_{1}, E_{1})$ and $T_{2}= (V_{2},E_{2})$ and an edge $h \in E_{2}$ such that $V_{1}\cap V_{2}=\left\{0\right\}$, $E_{1}\cup E_{2} = E$ and
	\begin{equation}\label{v1}
	{\mathrm{MaxFlow}}_{T_{1}}(0 \rightarrow \infty, \langle |J_{e}| \rangle) \geq {\mathrm{MaxFlow}}_{T_{2}}(0 \rightarrow \infty, \langle |J_{e}| \rangle) > 0
	\end{equation}
	and for every $\epsilon > 0$
	\begin{equation}\label{v2}
	{\mathrm{MaxFlow}}_{T_{2}}(0 \rightarrow \infty, \langle |J_{e}| \rangle) <
	{\mathrm{MaxFlow}}_{T_{2}}(0 \rightarrow \infty, \langle |J_{e}| + \epsilon \mathbbm{1}_{\left\{e=h\right\}} \rangle)
	\end{equation}
	where we think of the tree $T$ in such a way, that all edges in the trees $T_{1}$ and $T_{2}$ are pointing away from $0$.
\end{lemma}

\begin{figure}
	\[\begin{tikzpicture}[scale = 1., y={1cm/1.2}]
	\vertex[fill] (0) at (0,0) [label=below:$\tilde{0}$] {};
	
	\vertex[fill] (11) at (-0.75,1) {};
	\vertex[fill] (12) at (0.6,1) {};
	
	\vertex[fill] (21) at (-1.2,2) {};
	\vertex[fill] (22) at (1.2,2) {};
	
	\vertex[fill] (31) at (-1.6,3) {};
	\vertex[fill] (32) at (-0.8,3) {};
	\vertex[fill] (33) at (0.6,3) {};
	\vertex[fill] (34) at (1.6,3) {};
	
	\vertex[fill] (41) at (-1.6,4) {};
	\vertex[fill] (42) at (-0.8,4) {};
	\vertex[fill] (43) at (0.5,4) {};
	\vertex[fill] (44) at (1.2,4) {};
	\vertex[fill] (45) at (1.8,4) {};
	
	\vertex[draw=none] (i1) at (0.1,0.4) {};
	\vertex[draw=none] (i2) at (-0.5,4.4) {};
	
	\path
	(0) edge (11)
	(0) edge (12)
	
	(11) edge [line width = 1.3pt] node[left=0] {$\phi(\tilde{n})$} (21)
	(12) edge (22)
	
	(21) edge (31)
	(21) edge (32)
	(22) edge (33)
	(22) edge (34)
	
	(31) edge (41)
	(32) edge (42)
	(33) edge (43)
	(33) edge (44)
	(34) edge (45)
	
	(i1) edge[draw=lightgray, bend right=5] node[left] {\textcolor{gray}{$\tilde{T_2}$}} node[right] {\textcolor{gray}{$\tilde{T_1}$}} (i2)
	;

	\vertex[fill] (0) at (+7,0) {};
	
	\vertex[fill] (11) at (-0.75+7,1) [label=below:$0$] {};
	\vertex[fill] (12) at (0.6+7,1) {};
	
	\vertex[fill] (21) at (-1.2+7,2) {};
	\vertex[fill] (22) at (1.2+7,2) {};
	
	\vertex[fill] (31) at (-1.6+7,3) {};
	\vertex[fill] (32) at (-0.8+7,3) {};
	\vertex[fill] (33) at (0.6+7,3) {};
	\vertex[fill] (34) at (1.6+7,3) {};
	
	\vertex[fill] (41) at (-1.6+7,4) {};
	\vertex[fill] (42) at (-0.8+7,4) {};
	\vertex[fill] (43) at (0.5+7,4) {};
	\vertex[fill] (44) at (1.2+7,4) {};
	\vertex[fill] (45) at (1.8+7,4) {};
	
	\vertex[draw=none] (i1) at (-0.7+7,1.1) {};
	\vertex[draw=none] (i2) at (0+7,4.5) {} ;
	
	\path
	(0) edge (11)
	(0) edge (12)
	
	(11) edge [line width = 1.3pt] node[left=0] {$\phi(\tilde{n})$} (21)
	(12) edge (22)
	
	(21) edge (31)
	(21) edge (32)
	(22) edge (33)
	(22) edge (34)
	
	(31) edge (41)
	(32) edge (42)
	(33) edge (43)
	(33) edge (44)
	(34) edge (45)
	
	(i1) edge[draw=lightgray, bend right=15] node[left] {\textcolor{gray}{$T_2$}} node[right] {\textcolor{gray}{$T_1$}} (i2)
	;
	
	\vertex[draw=none] (p1) at (2,2) {};
	\vertex[draw=none] (p2) at (4.8,2) {};
	
	\path
	(p1) edge[->, line width = 1.3pt] (p2)
	;
	
	\end{tikzpicture}\]
	\centering
	\caption{Different separations of $T$ into two subtrees. $\phi(\tilde{n})$ is the bold edge.}
	\label{Fig3}
\end{figure}
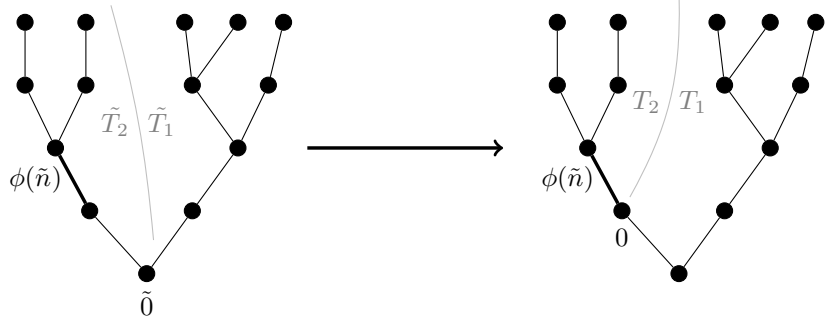

\begin{proof}
	We start with any separation of the tree $T$ into two subtrees $\tilde{T_1}=(\tilde{V_1},\tilde{E_1})$ and $\tilde{T_2}=(\tilde{V_2},\tilde{E_2})$ satisfying $E=\tilde{E_1}\cup\tilde{E_2}$, $\tilde{V_1}\cap\tilde{V_2}=\left\{\tilde{0}\right\}$ and
	\begin{equation*}
	{\mathrm{MaxFlow}}_{\tilde{T_1}}(\tilde{0} \rightarrow \infty, \langle |J_{e}| \rangle) \geq {\mathrm{MaxFlow}}_{\tilde{T_2}}(\tilde{0} \rightarrow \infty, \langle |J_{e}| \rangle) > 0 \ .
	\end{equation*} 
	Let $\phi : \N \rightarrow \tilde{E_2}$ be a bijective enumeration of $\tilde{E_2}$ which starts with all edges adjacent to $0$, then all edges $e \in \tilde{E_2}$ s.t. $|e|=1$ and so on. So in particular $n \mapsto |\phi(n)|$ is increasing. Define
	\begin{equation*}
	L\coloneqq \liminf_{\Pi \rightarrow \infty} \sum_{e \in \Pi \cap \tilde{E_2}} |J_{e}| \text{  \ \ and \ \  }
	L_{f} \coloneqq \liminf_{\Pi \rightarrow \infty} \sum_{e \in \Pi_{\succeq f}} |J_{e}|
	\end{equation*}
	for $f \in \tilde{E_2}$. $L$ and $L_{f}$ are constant almost surely by Kolmogorov's 0-1-law and $L>0$ by assumption. Furthermore
	\begin{equation*}
	\sum_{f\in \tilde{E_2}: |f|=k} |L_{f}| = L
	\end{equation*}
	for every $k \in \N$. Now define the function $MF:\N\rightarrow\R_{\ge 0}$ by
	\begin{equation*}
	MF(n) \coloneqq \mbox{MaxFlow}_{\tilde{T_2}}(0\rightarrow \infty, \langle |J_{e}| + \infty \cdot \mathbbm{1}\left\{\phi^{-1}(e)\leq n\right\} \rangle ) \ .
	\end{equation*}
	This means, we set the capacities at the edges $\left\{\phi(1) , ... , \phi(n)\right\}$ to $\infty$ and $MF(n)$ is the Maximal Flow in $\tilde{T_2}$ with respect to the new capacities. $MF$ is an increasing function, see Lemma \ref{L9}, bounded by $L$ and $MF(n) \rightarrow L$ for $n \rightarrow \infty$. We now want to show that $MF(0) = MaxFlow_{\tilde{T_2}}(0\rightarrow \infty, \langle |J_{e}|\rangle) < L$ almost surely. To see this, note that $MF(0) = L$ implies $|J_{f}| \geq L_{f} \ \forall f \in \tilde{E_2}$. Now let $c > 0$ be such that $\nu((-\epsilon,\epsilon)) \geq c\cdot\epsilon$ for small enough $\epsilon$. Then $\sum_{f\in \tilde{E_2}} \mathbb{P}\left(|J_{f}| < L_{f}\right) = \infty$: If $L_{f} \nrightarrow 0$ for $|f| \rightarrow \infty$ this is true, as all $J_{f}$ have the same distribution. In the case $L_f \rightarrow 0$ for $|f| \rightarrow \infty$
	\begin{equation*}
	\sum_{f\in \tilde{E}_{2}} \mathbb{P}\left(|J_{f}| < L_{f}\right) = \sum_{n=0}^{\infty} \ 
	\sum_{\substack{|f|=n \\ f \in \tilde{E}_2}}  \mathbb{P}\left(|J_{f}| < L_{f}\right) \geq
	\sum_{n=k}^{\infty} \ \sum_{\substack{|f|=n \\ f \in \tilde{E}_2}}  c \cdot L_{f} = \sum_{n=k}^{\infty} c\cdot L = \infty
	\end{equation*}
	for some $k$ large enough. Hence we obtain $MF(0)<L$ by a Borel-Cantelli-argument and independence of the $J_{f}$$'s$.\\
	\noindent
	Let $\tilde{n}$ be the smallest integer such that $MF(\tilde{n})>MF(0)$. As $\tilde{n}$ is the smallest integer, $|J_{g}| \geq |J_{\phi(\tilde{n})}| \ \forall g \in \mathcal{P}_{\phi(\tilde{n})}$, so actually we can choose $0$ as the vertex adjacent to $\phi(\tilde{n})$ and nearer to $\tilde{0}$. Then (\ref{v1}) and (\ref{v2}) hold true when one considers the subtrees $T_{\succeq \phi(\tilde{n})}$ and $T\setminus T_{\succeq \phi(\tilde{n})}$ with appropriate edge and vertex sets. A picture of our construction is given in Figure \ref{Fig3}.\\
\end{proof}

\begin{corollary}\label{infgs}
	Assume that the conditions of Lemma \ref{L2} hold. Then $|\mathcal{G}(J)|=\infty$ almost surely. Hence $|\mathcal{G}(J)| $ is either 2 or infinity almost surely for every tree and every distribution of linear growth.
\end{corollary}
\begin{proof}
	We have to show that there exist even infinitely many such divisions into two subtrees and respective edges $h$ satisfying (\ref{v1}) and (\ref{v2}). We can apply the construction of the proof of Lemma \ref{L2} also to the tree $\tilde{T}_{1}$ instead of $T$ and get subtrees $\tilde{T}_{1,1}$ and $\tilde{T}_{1,2}$ of $\tilde{T}_{1}$ and an edge $h_{1} \in \tilde{E}_{1,2}$, such that $\tilde{T}_{1,1}$ is the tree connected to the root $0$ and (\ref{v1}) and (\ref{v2}) hold true in $\tilde{T}_{1}$. By iterating this idea, we get the existence of infinitely many partitions of $T$ and edges $h$ satisfying (\ref{v1}) and (\ref{v2}). But as every such edge corresponds to a uniquely defined ground state pair (the one where $h$ is the only unsatisfied edge), we get $|\mathcal{G}(J)| = \infty$.
\end{proof}

Actually it suffices to require $\nu\left(\left\{0\right\}\right)=0$ and the lower bound of the linear growth condition. It seems plausible that Lemma \ref{L2}, and hence the equivalence of $i)$ and $ii)$ of Theorem \ref{T1} even hold true, as soon as $\nu((-\epsilon,\epsilon)) > 0 \ \forall \epsilon > 0$ and $\nu(\left\{0\right\})=0$, but there is no proof known to us.

\begin{corollary}
	Let T be a tree s.t. $p_{c} < 1$, where $p_{c}$ denotes the critical probability for bond percolation. Let $\nu$ be a probability measure on $\R$ such that $\nu((-\epsilon,\epsilon)) > 0$ and $\nu(\left\{0\right\})=0$ for every $\epsilon>0$. Then $|\mathcal{G}(J)| = \infty$ $\nu^{E}$-a.s.
\end{corollary}
\begin{proof}
	As $p_{c}<1$ there exists some $\epsilon > 0$ and infinitely many subtrees $\tilde{T}=(\tilde{V},\tilde{E})$ satisfying $|J_{e}|>\epsilon \ \forall e \in \tilde{E}$, hence $L=\infty$ for such a tree. By the same arguments as in the proof of Lemma \ref{L2} and Corollary \ref{infgs} we get that $|\mathcal{G}(J)| = \infty$ almost surely.
\end{proof}

\subsection{Random Walks and Maximal Flows for exponential couplings}\label{2.2}

In this section we prove a one-to-one correspondence between the maximal flow and recurrence/transience of random walks. The proof is based on the following theorem by R. Lyons, R. Pemantle and Y. Peres, see \cite{lyons1999resistance}.

\begin{theorem}\label{T4}
	Let G be a finite graph and $\kappa_{e}$ be independent exponentially-distributed random variables with mean $c_{e}$ and Z $\subset$ V, $0\in$ V. Then
	\begin{equation}\label{MF1}
	\mathbb{E}\left[{\mathrm{MaxFlow}}(0 \rightarrow Z, \langle \kappa_{e} \rangle)\right] \geq
	{\mathrm{Conduct}}(0 \rightarrow Z, \langle c_{e} \rangle) \ .
	\end{equation}
	
	Furthermore, if G is a tree, $0$ its root and $Z$ its leaves, then
	
	\begin{equation}\label{MF2}
	\mathbb{E}[{\mathrm{MaxFlow}}(0 \rightarrow Z, \langle \kappa_{e} \rangle)] \leq
	2{\mathrm{Conduct}}(0 \rightarrow Z, \langle c_{e} \rangle) \ .
	\end{equation}
\end{theorem}

Before going to the proof of Theorem \ref{T4} we need to prove the following Lemma, see also \cite{lyons1999resistance}. For the sake of completeness we will repeat the proofs by Lyons, Pemantle and Peres of Theorem \ref{T4} and Lemma \ref{L5}, but just for trees, where notation is a bit simpler.

\begin{lemma}\label{L5}
	Let $\theta$ be a flow from 0 to Z. Then there exists a measure $\mu$ on self-avoiding paths from $0$ to $Z$ so that
	\begin{equation}\label{Paths}
	\forall e \in E \text{    } \sum_{\mathcal{P} : e \in \mathcal{P}} \mu\left(\mathcal{P}\right) = \theta(e) \ .
	\end{equation}
\end{lemma}
\begin{proof}
	We use induction on the number of edges satisfying $\theta(e) \neq 0$. For $n=0$ the statement is clearly true. Now let $n+1$ be the number of edges satisfying $\theta(e) \neq 0$ and let $\mathcal{P}$ be a self-avoiding path from $0$ to $Z$ satisfying $\alpha \coloneqq \min_{e \in \mathcal{P}} \theta(e) > 0$. Let $\theta_{1}$ be the unit flow along $\mathcal{P}$. Then $\theta_{2} = \theta - \alpha \cdot \theta_{1}$ is also a flow from $0$ to $Z$ with number of edges satisfying $\theta_{2}(e) \neq 0$ less or equal than $n$. So we can find a measure $\mu_{2}$ satisfying (\ref{Paths}) for $\theta_{2}$ instead of $\theta$. But now the measure $\mu \coloneqq \mu_{2} + \alpha \cdot \delta_{\mathcal{P}}$ has the desired property (\ref{Paths}) for $\theta$.
\end{proof}

With this we are now ready to prove Theorem \ref{T4}.

\begin{proof}
	Let $\theta$ be the current flow of strength 1 from 0 to $Z$ and let $V$ be the associated voltage function which satisfies $V(z)=0 \ \forall z \in Z$. Let $\mu$ be a measure on paths from 0 to $Z$ such that (\ref{Paths}) holds. Since $\theta$ is a unit flow $\mu$ is a probability measure. Define a new flow $\psi$ by
	\begin{equation*}
	\psi(f) \coloneqq \sum_{f \in \mathcal{P}} \mu\left(\mathcal{P}\right)\min_{e\in \mathcal{P}}\frac{\kappa_{e}}{\theta(e)} \ .
	\end{equation*}
	$\psi$ is also a flow with respect to $\kappa_{e}$ since
	\begin{equation*}
	\psi(f)=
	\sum_{f \in \mathcal{P}} \mu\left(\mathcal{P}\right) \min_{e\in \mathcal{P}} \frac{\kappa_{e}}{\theta(e)} \leq \sum_{f \in \mathcal{P}} \mu\left(\mathcal{P}\right) \frac{\kappa_{f}}{\theta(f)} = \kappa_{f} \ .
	\end{equation*}
	Therefore,
	\begin{equation*}
	\mbox{MaxFlow}(0\rightarrow Z) \geq \sum_{\mathcal{P}} \mu\left(\mathcal{P}\right) \min_{e\in \mathcal{P}} \frac{\kappa_{e}}{\theta(e)} \ .
	\end{equation*}
	As the $\kappa_{e}$ are exponentially distributed and independent we have for all $s > 0$
	\begin{equation*}
	\mathbb{P}\left(\min_{e\in \mathcal{P}} \frac{\kappa_{e}}{\theta(e)} > s \right) =
	\prod_{e \in \mathcal{P}} \mathbb{P} \left(\kappa_{e} > \theta(e)s \right) =
	\prod_{e \in \mathcal{P}} e^{-\frac{\theta(e)s}{c_{e}}} = e^{-s\sum_{e\in\mathcal{P}}\frac{\theta(e)}{c _{e}}} \ .
	\end{equation*}
	Hence
	\begin{equation*}
	\mathbb{E}\left[\min_{e\in \mathcal{P}} \frac{\kappa_{e}}{\theta(e)}\right] = 
	\left[\sum_{e\in\mathcal{P}}\frac{\theta(e)}{c _{e}}\right]^{-1} =
	\left[\sum_{e\in\mathcal{P}}dV(e)\right]^{-1}  =
	\mbox{Conduct}(0\rightarrow Z, \langle c_{e} \rangle) \ .
	\end{equation*}
	where $dV((x,y))=V(x)-V(y)$ for $(x,y) \in E$ with $|x|<|y|$.\\
	For (\ref{MF2}) we distinguish two cases. If $deg(0) \geq 2$, (\ref{MF2}) is true by linearity of expectation, as we can split up the tree into two or more subtrees. If $deg(0)=1$ and $f=(0,a) \in E$ and $Z$ are the leaves of the tree, we can assume without loss of generality that $c_{f} = 1$, as Conductance, Expectation and MaxFlow are all linear under positive scalings. Furthermore
	\begin{align*}
	\mbox{MaxFlow}\left(0 \rightarrow Z, \langle \kappa_{e} \rangle\right) = \min \left\{\kappa_{f} , \mbox{MaxFlow} \left(a \rightarrow Z, \langle \kappa_{e} \rangle\right) \right\} \ .
	\end{align*}
	$X \coloneqq \mbox{MaxFlow}\left(a \rightarrow Z, \langle \kappa_{e} \rangle \right)$ and $Y \coloneqq \kappa_{f}$ are independent random variables, say on probability spaces $(\Omega_{1},\mu_{1})$ and $(\Omega_{2},\mu_{2})$. Let $C \coloneqq \mbox{Conduct}\left(a \rightarrow Z, \langle c_{e} \rangle \right)$. Then
	\begin{align*}
	\mathbb{E}\left[\mbox{MaxFlow}\left(0 \rightarrow Z, \langle \kappa_{e} \rangle \right) \right] &= 
	\int_{\Omega_{1}} \int_{\Omega_{2}} \min \left\{X(\omega_{1}) , Y(\omega_{2}) \right\} \mu_{2}(d\omega_{2}) \mu_{1}(d\omega_{1}) \\
	&=\int_{\Omega_{1}} 1-e^{-X(\omega_{1})} \mu_{1}(d\omega_{1}) 
	\leq 1-e^{-\mathbb{E}\left[X\right]} \leq 1-e^{-2C} \\
	&\leq 2\frac{C}{1+C} = 2 \mbox{Conduct}\left(0 \rightarrow Z, \langle c_{e} \rangle \right) \ .
	\end{align*}
	The first inequality follows by Jensen's inequality, the second inequality by the induction assumption. The last inequality is equivalent to
	\begin{equation*}
	(1-C)e^{2C} \leq 1+C \ .
	\end{equation*}
	For $C \geq 1$ this is clearly true, for $0 \leq C < 1$ the result is obtained by dividing by $1-C$ on both sides and developing the functions as power series. This concludes the proof.
\end{proof}

Having the theorem above at hand, we can prove the following:

\begin{corollary}\label{C6}
	Let T=(V,E) be a locally finite infinite tree and $(\kappa_{e})_{e \in E}$ be independent and exponentially distributed with mean 1. Let $\nu$ be the associated probability measure. Then the following are equivalent:\newline
	i) MaxFlow$(0 \rightarrow \infty, \langle \kappa_{e} \rangle)=0 \ \nu^{E}$-a.s.\newline
	ii) The simple random walk on T = (V,E) is recurrent a.s.
\end{corollary}

\begin{proof}
	Take $Z = V_{n} \coloneqq \left\{x \in V : |x| = n \right\}$. As we assume as always that $T$ is a tree without finite branches we get the inequalities
	\begin{equation*}
	\mathbb{E}[\mbox{MaxFlow}(0 \rightarrow V_{n}, \langle \kappa_{e} \rangle)] \geq
	\mbox{Conduct}(0 \rightarrow V_{n}, \langle 1 \rangle)
	\end{equation*}	
	and
	\begin{equation*}
	\mathbb{E}[\mbox{MaxFlow}(0 \rightarrow V_{n}, \langle \kappa_{e} \rangle)] \leq
	2\mbox{Conduct}(0 \rightarrow V_{n}, \langle 1 \rangle) \ .
	\end{equation*}
	Now take $n \rightarrow \infty$ on both sides. Note that 
	\begin{equation*}
	\mbox{MaxFlow}(0 \rightarrow V_{n}, \langle \kappa_{e} \rangle) \leq 
	\sum_{\{x \in V : 0 \rightarrow x\}}\kappa_{(0,x)} 
	\end{equation*}
	which has finite expectation. So by dominated convergence we can interchange limit and expectation on the left side of the inequalities. Hence we get 
	\begin{equation*}
	\mathbb{E}[\mbox{MaxFlow}(0 \rightarrow \infty, \langle \kappa_{e} \rangle)] \geq
	\mbox{Conduct}(0 \rightarrow \infty, \langle 1 \rangle)
	\end{equation*}	
	and
	\begin{equation*}
	\mathbb{E}[\mbox{MaxFlow}(0 \rightarrow \infty, \langle \kappa_{e} \rangle)] \leq
	2\mbox{Conduct}(0 \rightarrow \infty, \langle 1 \rangle) \ .
	\end{equation*}
	So if the Simple Random Walk is recurrent $\mathbb{E}[\mbox{MaxFlow}(0 \rightarrow \infty, \langle \kappa_{e} \rangle)] = 0$ and therefore $\mbox{MaxFlow}(0 \rightarrow \infty, \langle \kappa_{e} \rangle) = 0$ almost surely. If the Simple Random Walk is transient $\mathbb{E}[\mbox{MaxFlow}(0 \rightarrow \infty, \langle \kappa_{e} \rangle)] > 0$ and therefore $\mbox{MaxFlow}(0 \rightarrow \infty, \langle \kappa_{e} \rangle) > 0$ a.s..	
	
\end{proof}

\subsection{Maximal Flows for more general couplings}\label{2.3}

In the sections above we saw that there is a 1-to-1 correspondence between the maximal flow and the number of ground states for any distribution of linear growth and there is a connection between the maximal flow and recurrence/transience properties of the simple random walk on the tree, if $\nu$ is the law of an exponential distribution. The goal of this section is to prove Corollary \ref{C6} for all distributions of linear growth.

\begin{theorem}\label{T7}
	Let $\nu$ be a distribution of linear growth and let $(J_{e})_{e \in E}$ be i.i.d. with distribution $\nu$. Then the following are equivalent: \newline
	i) MaxFlow$(0 \rightarrow \infty, \langle |J_{e}| \rangle)=0 \ \nu^{E}$-a.s.\newline
	ii) The simple random walk on T = (V,E) is recurrent a.s.
\end{theorem}

With this theorem we can prove the following corollary:

\begin{corollary}\label{C8}
	Let $\nu$ be a distribution which is absolutely continuous with respect to the Lebesgue measure $\lambda$ on $\R$. Suppose that $f = \frac{d\nu}{d\lambda}$ is continuous at $0$ and $0 < f(0) < \infty$. Let $(J_{e})_{e \in E}$ be i.i.d. with law $\nu$. Then MaxFlow$(0 \rightarrow \infty, \langle |J_{e}| \rangle) = 0$ a.s. if and only if the simple random walk on T is recurrent a.s..
\end{corollary}
\begin{proof}
	Take $\delta > 0$ small enough such that $f((-\delta , \delta)) \subset (\frac{f(0)}{2} , 2f(0))$. Then for every $\epsilon < \delta$
	\begin{equation*}
	\nu((-\epsilon , \epsilon)) = \int_{-\epsilon}^{\epsilon} f(s) ds =
	\begin{cases}
	\geq \epsilon f(0) \\
	\leq 4\epsilon f(0)
	\end{cases} \ .
	\end{equation*}
	So $\nu$ is a distribution of linear growth. Now use Theorem \ref{T7} to conclude.
\end{proof}

Before proving Theorem \ref{T7} we have to deduce some properties of the maximal flow. 

\begin{lemma}\label{L9}
	Let $\mu$ and $(\kappa_{e})_{e \in E}$ be nonnegative real numbers. Then the following holds:\newline
	i) MaxFlow$(0 \rightarrow \infty, \langle \mu \kappa_{e} \rangle) = \mu \cdot$MaxFlow$(0 \rightarrow \infty, \langle \kappa_{e} \rangle)$ \newline
	ii) MaxFlow$(0 \rightarrow \infty, \langle \kappa_{e} \rangle) > 0$ if and only if MaxFlow$(0 \rightarrow \infty, \langle \lambda \kappa_{e} \rangle) >0$ for every $\lambda>0$\newline
	iii) Suppose $(\zeta_{e})_{e \in E}$ is another set of nonnegative real numbers satisfying $\kappa_{e} \leq \zeta_{e} \ \forall$ e $\in$ E. Then 
	\begin{equation*}
	{\mathrm{MaxFlow}}(0 \rightarrow \infty, \langle \kappa_{e} \rangle) \leq {\mathrm{MaxFlow}}(0 \rightarrow \infty, \langle \zeta_{e} \rangle)
	\end{equation*}
	iv) MaxFlow$(0 \rightarrow \infty, \langle \kappa_{e} \rangle) > 0$ if and only if MaxFlow$(0 \rightarrow \infty, \langle \kappa_{e}\wedge 1 \rangle) > 0$ 
\end{lemma}

\begin{proof}
	$i)$ and $iii)$ follow from the Max-Flow Min-Cut - Theorem (\ref{MFMC}) and positive homogenity of the infimum. $ii)$ follows directly from $i)$. Ad $iv)$:\newline \noindent Let $\theta$ be a flow with respect to $\kappa_{e}$. Let $\lambda > 0$ be small enough such that $ Strength(\lambda\theta) \leq 1$. Then $\lambda\theta(e)\leq 1 \ \forall e \in E$ and $\lambda\theta$ is a non zero flow with respect to $\kappa_{e}\wedge 1$. The other direction follows from $\kappa_{e}\wedge 1 \leq \kappa_{e}$ and $iii)$.	
\end{proof}

For the proof of Theorem \ref{T7} we use the quantile function. Let $X$ be a nonnegative random variable with distribution function $F(x) = \mathbb{P}(X\leq x)$. Then the function $Q:(0,1) \rightarrow \mathbb{R}$ defined by
\begin{equation*}
Q(p) = \inf\left\{x : p \leq F(x) \right\}
\end{equation*}
is called the quantile function of the random variable $X$. Now assume that $U$ is uniformly distributed on the interval (0,1). Then $Q(U)$ has the same distribution as $X$.\\

With this we are now ready to prove Theorem \ref{T7}.

\begin{proof}
	Let $(I_{e})_{e \in E}$ be i.i.d. with uniform distribution on the interval $(0,1)$. Let $F$ be the distribution function of $|J_{e}| \wedge 1$. Then $F(\epsilon) = \Theta(\epsilon)$ for $\epsilon \rightarrow 0$. Let $Q$ be the quantile function of $F$. Then we can find $0 < \lambda_{1}$ and $\lambda_{2} < \infty$ such that $\lambda_{1}s \leq Q(s) \leq \lambda_{2}s \ \forall s \in (0,1)$. $\lambda_{1}$ and $\lambda_{2}$ correspond to the grey lines in Figure \ref{Fig4}. Using $iii)$ in Lemma \ref{L9} we obtain
	\begin{equation*}
	\mbox{MaxFlow} (0 \rightarrow \infty, \langle \lambda_{1} I_{e} \rangle) \leq
	\mbox{MaxFlow} (0 \rightarrow \infty, \langle Q(I_{e}) \rangle) \leq
	\mbox{MaxFlow} (0 \rightarrow \infty, \langle \lambda_{2} I_{e} \rangle) \ .
	\end{equation*}
	As $(Q(I_{e}))_{e \in E}$ and $(|J_{e}| \wedge 1)_{e \in E}$ have the same distribution $MaxFlow (0 \rightarrow \infty, \langle I_{e} \rangle) = 0$ a.s. if and only if $\mbox{MaxFlow} (0 \rightarrow \infty, \langle |J_{e}| \wedge 1 \rangle) = 0$ a.s., which is equivalent to $\mbox{MaxFlow} (0 \rightarrow \infty, \langle |J_{e}| \rangle) = 0$ a.s., due to $iv)$ in Lemma \ref{L9}. So by applying this argument twice we get that $\mbox{MaxFlow} (0 \rightarrow \infty, \langle |J_{e}| \rangle) = 0$ a.s. if and only if $\mbox{MaxFlow} (0 \rightarrow \infty, \langle \kappa_{e} \rangle) = 0$ a.s., where $(\kappa_{e})_{e\in E}$ are i.i.d. exp(1)-distributed, which is equivalent to recurrence of the simple random walk, due to Corollary \ref{C6}.
\end{proof}

\begin{figure}
	\[\begin{tikzpicture}[scale = 3.]
	\vertex[draw=none] (y1) at (0,-0.2) {};
	\vertex[draw=none] (y2) at (0,1.2) [label=left:$F(x)$] {};
	\vertex[draw=none] (x1) at (-0.2,0) {};
	\vertex[draw=none] (x2) at (1.2,0) [label=below:$x$]{};
	
	\vertex[draw=none] (ym1) at (-0.06,1) {}; \vertex[draw=none] (ym2) at (0.06,1) [label=left:{\footnotesize 1}] {};
	\vertex[draw=none] (xm1) at (1,-0.06) [label=below :{\footnotesize 1}] {}; \vertex[draw=none] (xm2) at (1,0.06) {};
	
	\draw[thick,color=black] plot[samples=200, domain=0:0.3] 
	(\x,{(\x+0.5)^4-0.0625});
	\draw[thick,color=black] plot[samples=200, domain=0.3:0.5] 
	(\x,{0.5*(\x-0.3)+0.3471});
	\draw[thick,color=black] plot[samples=200, domain=0.5:1] 
	(\x,{7*(\x-0.75)^3+0.75});
	\draw[thick,color=black] plot[samples=200, domain=1:1.2] 
	(\x,{1});

	\path
	(y1) edge[->] (y2)
	(x1) edge[->] (x2)
	
	(xm1) edge (xm2)
	(ym1) edge (ym2)
	
	;
	
	\vertex[draw=none] (qy1) at (0+2,-0.2) {};
	\vertex[draw=none] (qy2) at (0+2,1.2) [label=left:$Q(s)$] {};
	\vertex[draw=none] (qx1) at (-0.2+2,0) {};
	\vertex[draw=none] (qx2) at (1.2+2,0) [label=below:$s$]{};
	
	\vertex[draw=none] (qym1) at (-0.06+2,1) {}; \vertex[draw=none] (qym2) at (0.06+2,1) [label=left:{\footnotesize 1}] {};
	\vertex[draw=none] (qxm1) at (1+2,-0.06) [label=below :{\footnotesize 1}] {}; \vertex[draw=none] (qxm2) at (1+2,0.06) {};
	
	\draw[thick,color=black] plot[samples=200, domain=0:0.3] 
	({(\x+0.5)^4-0.0625+2},\x);
	\draw[thick,color=black] plot[samples=200, domain=0.3:0.5] 
	({0.5*(\x-0.3)+0.3471+2},\x);
	\draw[thick,color=black] plot[samples=200, domain=0.5:1.004] 
	({7*(\x-0.75)^3+0.75+2},\x);
	
	\draw[thick,color=black] plot[samples=200, domain=0.4465:0.640625] 
	(\x+2,{0.5});
	\draw[thick,color=black] plot[samples=200, domain=0.855:1] 
	(\x+2,{1});
	
	\draw[thick,color=lightgray] plot[samples=200, domain=0:1] 
	(\x+2,0.6*\x);
	\draw[thick,color=lightgray] plot[samples=200, domain=0:0.6] 
	(\x+2,2.2*\x);
	
	\path
	(qy1) edge[->] (qy2)
	(qx1) edge[->] (qx2)
	
	(qxm1) edge (qxm2)
	(qym1) edge (qym2)
	
	;
	
	\end{tikzpicture}\]
	\centering
	\caption{A typical distribution and quantile function}
	\label{Fig4}
	
\end{figure}
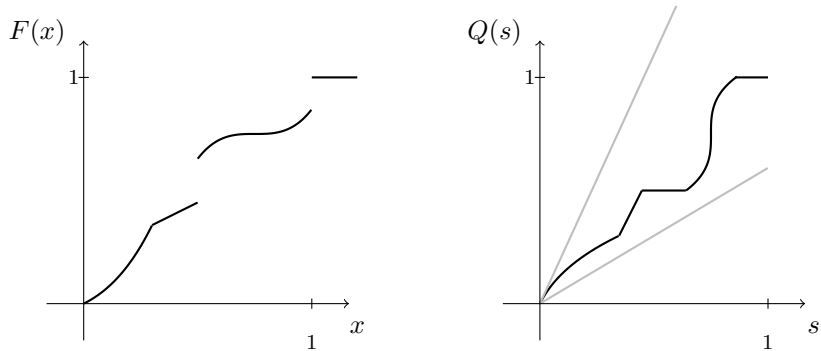

So we have seen that for distributions of linear growth there is a connection between the maximal flow and the conductance. It is a natural question to ask, whether this holds true for all absolutely continuous distributions with support at $0$. In fact it does not hold true. In chapter 3 we will give an example of a tree $T$ and an absolutely continuous distribution $\nu$, such that the simple random walk on $T$ is transient, but the maximal flow from 0 to $\infty$ with respect to some capacities $|J_{e}|$, which are i.i.d. with law $\nu$, equals 0 almost surely.

\subsection{Completing the proof}\label{2.4}

Up to now we did all proofs assuming that $T$ has no finite branches. We only assumed this for technical reasons, Theorem \ref{T1} holds true for every infinite tree. To see this, note that for any tree $T = (V,E)$ we can define a new tree $\tilde{T} = (\tilde{V},\tilde{E})$ by 
\begin{equation*}
\tilde{V} = \left\{x \in V : T_{\succeq x} \mbox{ contains infinitely many vertices } \right\}
\end{equation*}
and
\begin{equation*}
\tilde{E} = (\tilde{V} \times \tilde{V}) \cap E \ .
\end{equation*}
$\tilde{T}$ is another locally finite tree with $|\tilde{V}|=\infty$. It is the tree obtained by removing all finite branches from $T$, we will call $\tilde{T}$ also the {\sl backbone} of the tree. For some nonzero coupling values $(\kappa_{e})_{e \in E}$ the number of ground states for $T=(V,E)$ is exactly the number of ground states of the tree $\tilde{T} = (\tilde{V},\tilde{E})$ with coupling values $(\kappa_{e})_{e \in \tilde{E}}$. Removing the finite branches does not change the maximal flow from 0 to $\infty$, as the flow cannot enter any finite branch. Furthermore it does not change recurrence or transience of the simple random walk on the tree. To see this note that recurrence is equivalent to the existence of a finite energy flow from the root to $\infty$. Again this flow just lives on the backbone of the tree. So we can assume without loss of generality that $T$ does not have any finite branches.\newline
The equivalence of $i), ii)$ and $iii)$ of Theorem \ref{T1} has been proven in section \ref{2.1}, the equivalence of $iii)$ and $iv)$ in section \ref{2.3}. This concludes the proof.

\section{Dependence of $|\mathcal{G}(J)|$ on the coupling distribution}

So far we saw that the number of ground states is an almost sure constant for any tree $T$ and a coupling distribution of linear growth and either 2 or $\infty$ by Corollary \ref{infgs}. For the case of the half plane and for many coupling distributions it has been proven by L.-P. Arguin and M. Damron in \cite{arguin2014number} that $|\mathcal{G}(J)|$ is either 2 or $\infty$ almost surely. For graphs which have some translational symmetry it can be shown, see \cite{arguin2014number}, that the number of ground states is also an almost sure constant. In section \ref{3.1} we give an example of a tree $T=(V,E)$ and a distribution $\nu$, such that $\mathbb{P}\left(|\mathcal{G}(J)|=2\right) = \mathbb{P}\left(|\mathcal{G}(J)|=4\right) = \frac{1}{2}$. To achieve this we have to drop the condition $\nu((-\epsilon,\epsilon)) > 0$ for every $ \epsilon > 0$ .\\\noindent
In section \ref{3.2} we give an example of a tree $T$ with two different distributions such that the number of ground states is 2 almost surely or $\infty$ almost surely, depending on the coupling distribution. For this we have to drop the linear growth condition for one of the two distributions.

\subsection{Randomness of $|\mathcal{G}(J)|$}\label{3.1}

\begin{figure}
	\[\begin{tikzpicture}[scale = 0.4,- ,>= latex ]
	\vertex[fill] (l) at (-1,0) {};
	\vertex[fill] (r) at (1,0) {};
	\vertex[fill] (ro1) at (2.732,1) {};
	\vertex[fill] (ro2) at (4.464,2) {};
	\vertex[fill] (ru1) at (2.732,-1) {};
	\vertex[fill] (ru2) at (4.464,-2) {};
	\vertex[fill] (lo1) at (-2.732,1) {};
	\vertex[fill] (lo2) at (-4.464,2) {};
	\vertex[fill] (lu1) at (-2.732,-1) {};
	\vertex[fill] (lu2) at (-4.464,-2) {};
	
	\vertex[draw=none] (lu3) at (-6.196,-3) {};
	\vertex[draw=none] (ru3) at (6.196,-3) {};
	\vertex[draw=none] (lo3) at (-6.196,3) {};
	\vertex[draw=none] (ro3) at (6.196,3) {};
	
	\path
	(l) edge node[above=0] {f} (r)
	(l) edge (lo1)
	(l) edge (lu1)
	(r) edge (ro1)
	(r) edge (ru1)
	
	(lo1) edge (lo2)
	(lu1) edge (lu2)
	(ro1) edge (ro2)
	(ru1) edge (ru2)
	
	(lo3) edge[dotted] (lo2)
	(lu3) edge[dotted] (lu2)
	(ro3) edge[dotted] (ro2)
	(ru3) edge[dotted] (ru2)
	
	;
	\end{tikzpicture}\]
	\centering
	\caption{}
	\label{Fig5}
	
\end{figure}

Let $\nu$ be the uniform distribution on the interval $(1,3)$ and let $T=(V,E)$ be the tree of Figure \ref{Fig5}. We have one edge $f$ in the middle and both vertices adjacent to $f$ are starting points of two halflines going to $\infty$. Here the number of ground states depends on the coupling value $J_{f}$. If $\sigma$ is a ground state, then all edges $e \in E\setminus\left\{f\right\}$ have to be satisfied, as for every $e\in E\setminus{\left\{f\right\}}$ $J_{e}>1$ and there almost surely exists an edge $h$ in the same halfline such that $J_{h}<J_{e}$. Therefore $f$ is the only edge which can be satisfied or unsatisfied in a ground state. As all couplings are positive the natural ground states are the spin configurations satisfying either $\sigma_{x} = +1 \ \forall x\in V$ or $\sigma_{x} = -1 \ \forall x\in V$. \newline
\noindent
If $J_{f} > 2$ the natural ground states are the only ground states. When $f$ is not satisfied we can almost surely find edges $h_{1}$ in the upper right and $h_{2}$ in the lower right halfline such that $J_{h_{1}}+J_{h_{2}} < J_{f}$, which contradicts (\ref{gs}). So the natural ground states are the only ground states.\newline
\noindent
If $J_{f} \leq 2$ the spin configurations where $f$ is the only unsatisfied edge are ground states. The reason for this is that for $B \subset V$ and $f\in \partial B$ there are at least two other edges in $\partial B$. Note that this are precisely the spin configurations which are $+1$ on the right-hand side of the graph and $-1$ on the left-hand side, or vice versa. Additional to those ground states the natural ones still exist, so $|\mathcal{G}(J)|=4$ in this case.

\subsection{A tree with two different coupling distributions}\label{3.2}

The main goal of this chapter is to give a tree $T$ and two equivalent distributions, such that $|\mathcal{G}(J)|$ is 2 or $\infty$ almost surely, depending from which of the two distributions the couplings are drawn.\\

Let $T$ be a tree such that the simple random walk on $T$ is recurrent. Let $(\omega_{n})_{n \in \N}$ be a summable sequence of positive real numbers. Then, as already noted in (\cite{lyons2017probability}, Chapter 3)
\begin{equation}\label{csl1}
\inf\limits_{\Pi \ cutset} \ \sum_{e \in \Pi} \omega_{|e|} = 0 \ .
\end{equation}
Otherwise there would exist a nonzero flow from zero to infinity w.r.t. the capacities $\kappa_{e} = \omega_{|e|}$ and
\begin{equation*}
\sum_{e \in E} \theta(e)^{2} = \sum_{n=0}^{\infty} \ \sum_{e \in E_{n}} \theta(e)^{2} \leq \sum_{n=0}^{\infty} \omega_{n} \sum_{e \in E_{n}} \theta(e) = ||(\omega_{n})||_{\ell^{1}(\N)} \cdot  \mbox{strength}(\theta) < \infty
\end{equation*}
which contradicts recurrence of the simple random walk.	For some set of nonnegative coupling values $(\kappa_{e})_{e \in E}$ we define a new set of coupling values $(\kappa^{min}_{e})_{e \in E}$ by
\begin{equation*}
\kappa^{min}_{e} \coloneqq \min\left\{\kappa_{f} : f \in \mathcal{P}_{e} \right\}
\end{equation*}
where $\mathcal{P}_{e}$ is the path connecting $e$ and $0$. The idea is to find a condition of a similar form as (\ref{csl1}) which ensures that $MaxFlow (0 \rightarrow \infty, \langle \kappa_{e} \rangle) = 0$ a.s., where the $(\kappa_{e})_{e \in E}$ are i.i.d. and exponentially distributed with parameter 1.

\begin{lemma}\label{L10}
	Let $(\kappa_{e})_{e \in E}$ be nonnegative coupling values. Then
	\begin{equation*}
	\mbox{MaxFlow} (0 \rightarrow \infty, \langle \kappa_{e} \rangle) = \mbox{MaxFlow} (0 \rightarrow \infty, \langle \kappa^{min}_{e} \rangle) \ .
	\end{equation*}
	Even more is true: $\theta$ is a flow with respect to the capacities $\kappa_{e}$ if and only if it is a flow with respect to $\kappa^{min}_{e}$.
\end{lemma}
\begin{proof}
	If $\theta$ is a flow w.r.t. $\kappa^{min}_{e}$ then it is also a flow w.r.t. $\kappa_{e}$, as $\kappa^{min}_{e} \leq \kappa_{e} \ \forall e \in E.$
	For the converse direction let $f$ be an edge such that $f \in \mathcal{P}_{e}$ and $\kappa_{f} = \kappa^{min}_{e}$. Then $f \preceq e$ and hence $\theta(f) \geq \theta(e)$, so $\theta(e) \leq \theta(f) \leq \kappa_{f} = \kappa^{min}_{e}$ and hence $\theta$ is also a flow with respect to $\kappa^{min}_{e}$.
\end{proof}

\begin{theorem}\label{T11}
	Let $(\kappa_{e})_{e \in E}$ be i.i.d. with distribution $\nu$, where $\nu$ is the law of an exponentially distributed random variable with mean 1, and let T be a tree such that $\inf\limits_{\Pi \ cutset} \ \sum_{e \in \Pi}  \frac{1}{|e|} = 0$. Then MaxFlow$(0 \rightarrow \infty, \langle \kappa_{e} \rangle) = 0$ $\nu^{E}$-a.s. and hence all statements of Theorem \ref{T1} hold true.
\end{theorem}
\noindent
Using the equivalences proven in Theorem \ref{T1} this also follows from a special case of \cite{collevecchio2017branching}, but we give a different proof, as we will use the same technique in a slightly different setting again at a later point.
\begin{proof}
	Let $\epsilon > 0$ and let $\Pi$ be a cutset such that $\sum_{e \in \Pi}  \frac{1}{|e|} < \epsilon$. Then 
	\begin{align*}
	\mathbb{E}[&\mbox{MaxFlow}(0 \rightarrow \infty, \langle \kappa_{e} \rangle)] = \mathbb{E}[\mbox{MaxFlow}(0 \rightarrow \infty, \langle \kappa^{min}_{e} \rangle)] \leq  \mathbb{E} \left[ \ \sum_{e \in \Pi}  \kappa^{min}_{e} \ \right] \\ 
	&=\sum_{e \in \Pi}  \mathbb{E}[\kappa^{min}_{e}] = \sum_{e\in\Pi} \mathbb{E}\left[\min_{f\in\mathcal{P}_{e}}\kappa_{f}\right] = \sum_{e \in \Pi}  \frac{1}{|e|+1} < \epsilon \ .
	\end{align*}
	As $\epsilon$ was arbitrary $MaxFlow(0 \rightarrow \infty, \langle \kappa_{e} \rangle) = 0$ a.s..
\end{proof}
\noindent
One can extend Theorem \ref{T11} assuming the slightly weaker condition $\liminf\limits_{\Pi \rightarrow \infty} \sum_{e \in \Pi}  \frac{1}{|e|} < \infty$, but to prove this one first needs to deduce several other lemmas. The proof is given in the Appendix.\\

In Corollary \ref{infgs} we saw that $|\mathcal{G}(J)|$ is either 2 or $\infty$ and this is independent of the specific choice of the coupling distribution $\nu$, as long as $\nu$ is a distribution of linear growth. Below we show that $|\mathcal{G}(J)|$ still depends on the coupling distribution $\nu$. We do this by giving an example of a tree $T=(V,E)$ and i.i.d. random variables $(J_{e})_{e \in E}$ with absolutely continuous distribution $\nu$ with support at 0, such that the simple random walk on $T$ is transient but $MaxFlow(0 \rightarrow \infty, \langle |J_{e}| \rangle$) = 0 almost surely. Let $(I_{e})_{e \in E}$ be i.i.d. uniformly distributed on $(0,1)$. Then $|\mathcal{G}(I)| = \infty $ a.s., but $|\mathcal{G}(J)| = 2$ a.s..\\
\noindent
For the rest of this chapter we assume that $T$ is a spherical symmetric tree such that the simple random walk on $T$ is transient and $(n+1)^{2} \leq |E_{n}| \leq 2(n+1)^{2}$ for every $n\in\N$. To show existence of such a tree, note that if $\theta$ is the flow satisfying $\theta(e) = \frac{1}{|E_{|e|}|}$ for every $e \in E$
\begin{equation*}
\sum_{e \in E} \theta(e)^{2} = \sum_{n=0}^{\infty} \ \sum_{e \in E_{n}} \theta(e)^{2} = \sum_{n=0}^{\infty} |E_{n}| \cdot \frac{1}{|E_{n}|^{2}} \leq \sum_{n=0}^{\infty} \frac{1}{(n+1)^{2}} < \infty
\end{equation*}
which implies transience of the simple random walk on $T$.\\

\noindent
Let $\nu$ be the probability distribution on $\left[0,1\right]$ such that
\begin{equation*}
\nu([0,x]) = \sqrt[3]{x} \ \mbox{ for } 0\leq x \leq 1 \ .
\end{equation*}
The measure $\nu$ is absolutely continuous with respect to the Lebesgue measure and has density
\begin{equation*}
f(s) = 
\begin{cases}
\frac{1}{3} \ s^{-\frac{2}{3}} & \text{for} \ 0<s\leq 1\\
0 & \text{else}
\end{cases} \ .
\end{equation*}

Now we want to investigate how fast $\mathbb{E}\left[\min\limits_{1\leq i \leq n} X_{i}\right]$ tends to 0, when the $X_{i}$ are i.i.d. random variables with distribution $\nu$. To do so note first that for $s\in (0,1) $

\begin{equation*}
\mathbb{P}\left(\min\limits_{1\leq i \leq n} X_{i} > s \right) = \mathbb{P}\left(X_{1} > s\right)^{n} = \left(1-\sqrt[3]{s}\right)^{n} \ .
\end{equation*}
Hence
\begin{equation*}
\mathbb{E}\left[\min\limits_{1\leq i \leq n} X_{i}\right] = \int_{0}^{1} \mathbb{P}\left(\min\limits_{1\leq i \leq n} X_{i} > s\right) ds = \int_{0}^{1} \left(1-\sqrt[3]{s}\right)^{n}ds = \int_{0}^{1} (1-s)^{n}\cdot 3s^{2} ds \ .
\end{equation*}
Now use the following formula, see for example \cite{forrester2008importance}:

\begin{equation*}
\Gamma(\alpha) \Gamma(\beta) = \Gamma(\alpha + \beta)\cdot \int_{0}^{1} s^{\alpha - 1}(1-s)^{\beta - 1} ds
\end{equation*}
where $\Gamma$ is the Gamma function and both $\alpha$ and $\beta$ are positive. This implies

\begin{equation*}
\mathbb{E}\left[\min\limits_{1\leq i \leq n} X_{i}\right] = \frac{6 \Gamma(n+1)}{\Gamma(n+4)} \leq \frac{6}{n^{3}} \ .
\end{equation*}
Using the same steps as in the proof of Theorem \ref{T11} and using that the $(J_{e})_{e \in E}$ are i.i.d. with distribution $\nu$ we get that
\begin{equation*}
\mathbb{E}\left[\mbox{MaxFlow}(0 \rightarrow \infty, \langle J_{e} \rangle)\right] \leq \inf\limits_{\Pi cutset }\sum_{e \in \Pi}  \frac{6}{|e|^{3}} = 0
\end{equation*}
where the last equality follows by considering the cutsets $E_{n}$ and using that $|E_{n}| \leq 2(n+1)^{2}$. Hence, by Theorem \ref{T1} and Theorem \ref{T7}, we have constructed a tree, such that the natural ground states are the only ones, when the couplings have law $\nu$. However, when the couplings are uniformly distributed on the interval $(0,1)$, there are infinitely many by Corollary \ref{infgs}.

\section{Open problems}

We conclude the paper with three open problems: \\ \\
In section 2 we saw that the number of ground states is an almost sure constant for every tree $T$ and every distribution of linear growth and either 2 or $\infty$ by Corollary \ref{infgs}. However, there still are trees and distributions, where $|\mathcal{G}(J)|$ is a non degenerate random variable, see section \ref{3.1}. It remains unsolved, whether there exists a graph $G$ and a distribution $\nu$ satisfying $\nu((-\epsilon,\epsilon))>0 \ \forall \epsilon > 0$, where $|\mathcal{G}(J)|$ is not an a.s. constant or where $|\mathcal{G}(J)| \in \N \setminus \left\{2\right\}$ has positive probability.\\

\noindent
In section 2 we also saw that there is a connection between the number of ground states, percolation and random walks for various distributions. Is there also and connection between those for more general graphs, for example lattices? \\

\noindent
For any tree $T$
\begin{equation}\label{bn}
\inf\limits_{\Pi \ cutset} \ \ \sum_{e \in \Pi} \lambda^{-|e|} = 0 \ \ \forall \lambda > 1
\end{equation}
is a necessary and sufficient condition for ensuring $p_{c} = 1$. Furthermore we saw that 
\begin{equation*}
\inf\limits_{\Pi \ cutset} \ \ \sum_{e \in \Pi} \omega_{|e|} = 0 \ \ \ \forall \ \text{positive sequences} \ (\omega_{n})_{n \in \N} \in \ell^{1}(\N) 
\end{equation*}
is necessary and that
\begin{equation*}
\liminf\limits_{\Pi \rightarrow \infty} \ \ \sum_{e \in \Pi} \frac{1}{|e|} < \infty 
\end{equation*}
is sufficient to ensure recurrence of the simple random walk on the tree. It is still unknown, whether one can find a condition of the same form as (\ref{bn}), which is equivalent to recurrence of the simple random walk.\\

\textbf{Acknowledgements.} I would like to thank Noam Berger for introducing me to the theory of Spin Glasses and for many useful discussions and helpful comments on this paper. The results presented in this paper were part of my Bachelor thesis at TU München written under his supervision. Furthermore, I would like to thank an anonymous referee for many useful comments and remarks. This work is supported by TopMath, the graduate program of the Elite Network of Bavaria and the graduate center of TUM Graduate School.

\section{Appendix}

\begin{lemma}\label{L12}
	Let $(\kappa_{e})_{e \in E}$ be i.i.d. with distribution exp(1) and let $T$ be a tree. Then $MaxFlow(0 \rightarrow \infty, \langle \kappa_{e} \rangle)$ is either $0$ almost surely, or it has unbounded support.
	
\end{lemma}
\begin{proof}
	In the same way as we did in the proof of Corollary \ref{C6} we get that
	\begin{equation*}
	\mathbb{E}[\mbox{MaxFlow}(0 \rightarrow \infty, \langle J_{e} \rangle)] \geq
	\mbox{Conduct}(0 \rightarrow \infty, \langle c_{e} \rangle)
	\end{equation*}
	where the $J_{e}$ are independent exponentially distributed with mean $c_{e}$. Now note that we have by Thomson's Principle, see for example (\cite{lyons2017probability}, Chapter 2)
	\begin{equation*}
	\mbox{Resist}(0\rightarrow \infty, \langle c_{e} \rangle) = \min\left\{ \sum_{e \in E} \theta(e)^{2}c_{e}^{-1} : \theta \ \text{unit flow from 0 to} \ \infty \right\} \ .
	\end{equation*}
	So if $Resist(0\rightarrow \infty, \langle 1 \rangle) < \infty$, i.e. if the simple random walk is transient, for every $\delta > 0$ we can find some finite set $F^{\delta} \subset E$ and some $M > 0$ big enough such that 
	\begin{equation}\label{delta}
	\mbox{Resist}(0\rightarrow \infty, \langle C^{M,F^{\delta}}_{e} \rangle) < \delta
	\end{equation}
	where
	\begin{equation*}
	C^{M,F^{\delta}}_{e} =
	\begin{cases}
	M & e \in F^{\delta}\\
	1 & \mbox{else}
	\end{cases} \ .
	\end{equation*}
	Now let $(\kappa_{e})_{e \in E}$ be i.i.d. with law exp(1) and let $(I_{e})_{e \in E}$ be independent exponentially distributed random variables with mean $C^{M,F^{\delta}}_{e}$. From (\ref{delta}) and (\ref{MF1}) we get that
	\begin{equation*}
	\mathbb{E}[\mbox{MaxFlow}(0 \rightarrow \infty, \langle I_{e} \rangle)] \geq 
	\mbox{Conduct}(0\rightarrow \infty, \langle C^{M,F^{\delta}}_{e} \rangle)^{-1} > \delta^{-1}
	\end{equation*}
	and hence
	\begin{equation*}
	\mathbb{P}(\mbox{MaxFlow}(0 \rightarrow \infty, \langle I_{e} \rangle) > \delta^{-1}) > 0 \ .
	\end{equation*}
	As $F^{\delta}$ is finite the laws of $(\kappa_{e})_{e \in E}$ and $(I_{e})_{e \in E}$ are equivalent. Therefore 
	\begin{equation*}
	\mathbb{P}(\mbox{MaxFlow}(0 \rightarrow \infty, \langle \kappa_{e} \rangle) > \delta^{-1}) > 0 \ .
	\end{equation*}
	As $\delta$ was arbitrary we have that $MaxFlow(0 \rightarrow \infty, \langle \kappa_{e} \rangle)$ is unbounded as soon as the simple random walk is transient.
	If the simple random walk is recurrent, we know from Corollary \ref{C6} that $MaxFlow(0 \rightarrow \infty, \langle \kappa_{e} \rangle) = 0$ a.s..
\end{proof}

\begin{lemma}\label{L14}
	Let $T$ be a tree such that $p_{c} = 1$. Let $(\kappa_{e})_{e \in E}$ be i.i.d. random variables with exponential distribution and mean 1. Then
	\begin{equation}\label{liminf}
	\liminf_{\Pi \rightarrow \infty} \sum_{e \in \Pi} \kappa_{e}  = 
	\liminf_{\Pi \rightarrow \infty} \sum_{e \in \Pi} \kappa_{e}^{min}  \ \ \ {\mbox \ a.s.}
	\end{equation}
\end{lemma}

\begin{proof}
	The greater or equal in (\ref{liminf}) is clearly true, as $\kappa_{e}^{min} \leq \kappa_{e}$. For the other direction let $M \in \R \cup \infty$ be such that $\liminf_{\Pi \rightarrow \infty} \sum_{e \in \Pi} \kappa_{e}^{min} \leq M$. Now define the function $\Xi: E \rightarrow E$ by
	\begin{equation*}
	e \mapsto \arg\min\left\{\kappa_{f} : f \in \mathcal{P}_{e} \right\}
	\end{equation*}
	$\Xi(e)$ is in general not uniquely defined, so among all minimizers we will always take the one which is closest to $e$ and hence furthest from the root. For some cutset $\Pi$ the set $\Pi^{min} \coloneqq \Xi(\Pi)$ is also a cutset and
	\begin{equation*}
	\sum_{e \in \Pi^{min}} \kappa_{e} \leq  \sum_{e \in \Pi} \kappa_{e}^{min} \ .
	\end{equation*}
	Note that we do not have equality in general, as $\Xi$ does not have to be injective. So we have to show that $\Pi^{min}\rightarrow \infty$, a soon as $\Pi \rightarrow \infty$. Therefore assume that we have a sequence of cutsets $\Pi_{n} \rightarrow \infty$, but $\Pi_{n}^{min}\nrightarrow \infty$. This is equivalent to the existence of some $e \in E$ such that $|\Xi^{-1}(e)| = \infty$. Consider the tree $G = (V_{e},\Xi^{-1}(e))$, where $V_{e}$ is the set of vertices which are adjacent to at least one edge in $\Xi^{-1}(e)$. This is an infinite subtree of $T$ satisfying $\kappa_{f}<\kappa_{e} \ \forall f \in \Xi^{-1}(e)$. Such a tree occurs with probability 0, as $p_{c} = 1$, hence we get that there exists almost surely no $e \in E$ such that $|\Xi^{-1}(e)| = \infty$. This implies that $\Pi_{n}^{min}\rightarrow \infty$ a.s. and $\liminf_{\Pi \rightarrow \infty} \sum_{e \in \Pi} \kappa_{e} \leq M$ almost surely. As $M$ was arbitrary, we directly get (\ref{liminf}).
\end{proof}

The left-hand side of (\ref{liminf}) is constant almost surely by Kolmogorov's 0-1-law. Lemma \ref{L14} directly implies that the right-hand side is also an almost sure constant. In the proof of the desired theorem we use the concept of {\sl branching numbers}:\\

Let $T$ be a tree with root 0. Define the branching number $br T$ of $T$ by
\begin{equation*}
br T \coloneqq \inf\left\{\lambda>0 : \inf\limits_{\Pi \ cutset} \ \sum_{e \in \Pi} \lambda^{-|e|} = 0 \right\} \ .
\end{equation*}
\noindent
There is a connection between bond percolation and the branching number of a tree, which is due to R. Lyons, see \cite{lyons1992random} for details. Let $p_{c}$ be the critical probability for bond percolation on the tree $T$. Then
\begin{equation*}
p_{c} = \frac{1}{br T} \  .
\end{equation*}

\begin{theorem}\label{T15}
	Let T be a tree such that 
	\begin{equation}\label{bdd}
	\liminf\limits_{ \Pi \rightarrow \infty} \sum_{e \in \Pi} \frac{1}{|e|} < \infty \ .
	\end{equation}
	Then the simple random walk on $T$ is recurrent.
\end{theorem}

\begin{proof}
	First note that (\ref{bdd}) implies $p_{c} = brT^{-1} = 1$. Let $(\kappa_{e})_{e \in E}$ be i.i.d. exponentially distributed with mean 1. Note that for every positive random variable $X$ and any constant $M \geq \mathbb{E}[X]$
	\begin{equation*}
	\mathbb{P}(X \leq 2M) \geq \frac{1}{2} \ .
	\end{equation*}
	So for any $C \in \R$ and any sequence of cutsets $\Pi_{n}$ satisfying $\Pi_{n} \rightarrow \infty$ for $n \rightarrow \infty$ and $\sum_{e \in \Pi_{n}} \frac{1}{|e|} \leq C \ $ for every $n\in\N$ 
	\begin{align*}
	\mathbb{P}\left(\sum_{e \in \Pi_{n}} \kappa^{min}_{e} \leq 2C \ \text{for infinitely many }n\right) &= \mathbb{P}\left(\ \bigcap_ {k=1}^{\infty} \ \bigcup_{n=k}^{\infty} \ \left\{\sum_{e \in \Pi_{n}} \kappa^{min}_{e} \leq 2C \right\} \right) \\
	& \geq \limsup_{n\rightarrow \infty}\mathbb{P}\left(\sum_{e \in \Pi_{n}} \kappa^{min}_{e} \leq 2C \right) \geq \frac{1}{2} \ .
	\end{align*}
	As $\mathbb{P}\left(\liminf\limits_{ \Pi \rightarrow \infty} \sum_{e \in \Pi} \kappa_{e}^{min} \leq M\right) \in \left\{0,1\right\}$ for any $M\in \R$ applying Lemma \ref{L14} yields
	\begin{equation*}
	\mathbb{P}\left(\liminf\limits_{ \Pi \rightarrow \infty} \sum_{e \in \Pi} \kappa_{e} \leq 2C\right) = 1 \ .
	\end{equation*}
	But this implies that $MaxFlow(0 \rightarrow \infty, \langle \kappa_{e} \rangle)$ is bounded by $2C$. By Lemma \ref{L12} we obtain $MaxFlow(0 \rightarrow \infty, \langle \kappa_{e} \rangle) = 0$ almost surely. Therefore the simple random walk on $T$ is recurrent, by Corollary \ref{C6}.\\
\end{proof}

\end{document}